\documentclass[11pt]{article}

\usepackage{amsmath}
\usepackage{amssymb}
\usepackage{amsthm} 
\usepackage{enumerate}

\usepackage{hyperref} 

\usepackage{graphicx}

\usepackage{url}

\usepackage{scalerel}       %for making \bo  and \di

\usepackage{mathrsfs}       % provides script font by \mathscr

\usepackage{newtxtext}       %
\usepackage{newtxmath}       % selects Times Roman as basic font

\theoremstyle{plain}
\newtheorem{theorem}{Theorem}[]
\newtheorem{lemma}[theorem]{Lemma}

\newtheorem{corollary}[theorem]{Corollary}

\theoremstyle{definition}

\numberwithin{equation}{section}

\def\a{\alpha}
\def\b{\beta}

\def\d{\delta}
\def\ka{\kappa}
\def\LL{\Lambda}
\def\om{\omega}
\def\ph{\varphi}
\def\Sig{\Sigma}
\def\th{\theta}

\def\ab#1{|#1|}

\def\C{\mathscr{C}}

\def\F{\mathscr{F}}

\def\G{\mathscr{G}}

\def\M{\mathscr{M}}

\def\N{\mathscr{N}}

\def\bo{\mathop{\scaleto{\Box}{8pt}}}       %Box
\def\di{\mathop{\scaleto{\Diamond}{11pt}}}     %Diamond

\def\sub{\subseteq}

\def\<{\langle}  
\def\>{\rangle}

\usepackage{fancyhdr}
\fancypagestyle{firststyle}
{
   \fancyhf{}
}

\title{Completeness of Pledger's modal logics of one-sorted projective and elliptic planes}
\author{Robert Goldblatt\\
	    School of Mathematics and Statistics\\
	    Victoria University of Wellington }
\date{}                                           % Activate to display a given date or no date

\begin{document}
\maketitle

\begin{abstract}
Ken Pledger devised a one-sorted approach to the incidence relation of plane geometries, using structures that also support models of propositional modal logic. He introduced a modal system 12g that is valid in one-sorted projective planes, proved that it has finitely many non-equivalent modalities, and identified all possible modality patterns of its extensions. One of these extensions 8f is valid in elliptic planes. These results were presented in his  1980 doctoral dissertation \cite{pled:some81}, which has been reprinted in the {\em Australasian Journal of Logic}, 
   vol.~18, no.~4.
\url{https://doi.org/10.26686/ajl.v18i4.6831}

Here we show that 12g and 8f are strongly complete for validity in their intended one-sorted geometrical interpretations, and have the finite model property. The proofs apply standard technology of modal logic (canonical models, filtrations) together with a step-by-step procedure introduced by Yde Venema for constructing two-sorted projective planes.

\end{abstract}



\thispagestyle{firststyle} % invoke first page style

%%%%%%%%%%%%%%% %%%%%%%%%%%%%%%%%%%%%%%%%%%%%%%%%%%%% Introduction %%%%%%%%%%%%%%%%%%%%%%%%%%%%%%%%%%%%%%%%%%%%%%%%%%%%%%%

\section{Pledger's dissertation}   \label{sec:diss}

Kenneth Ernest Pledger (1938--2016) was a mathematician who taught at the Victoria University of Wellington for 52 years. 
He was primarily interested in geometry and the history of mathematics, but participated regularly in logic seminars conducted jointly by mathematicians and philosophers at VUW. This led him to make some contributions to modal logic, including the articles \cite{pled:mod72,pled:some75,pled:loca80}. In the mid-1970's he independently discovered the fact, of significance to provability logic, that any normal modal logic containing the L\"ob axiom $\bo(\bo \ph\to \ph)\to\bo \ph$ must extend the logic K4, i.e.\ must contain the transitivity axiom $\bo \ph\to\bo\bo \ph$ (see \cite[p.~11]{bool:logi93}).
He also showed that any normal modal logic containing the Grzegorczyk axiom
$\bo(\bo(\ph\to\bo \ph)\to \ph)\to \ph$
is an extension of S4 \cite[pp.~157, 259]{bool:logi93}.

In 1981 Pledger was awarded the PhD degree by the University of Warsaw for a dissertation \cite{pled:some81} entitled
\emph{Some interrelations between geometry and modal logic}, with Les\l aw Szczerba (1938--2010) as promoter. Their relationship had developed from a sabbatical visit made by Szczerba in 1977  to Wellington, followed by one made  by Pledger in 1979 to Warsaw, where the University's regulations for doctoral study did not include a residence requirement. Szczerba was a former student of Wanda Szmielew. He worked mainly on the foundations of geometry, including in collaboration with Tarski \cite{szcz:meta79}.

Pledger's dissertation set out a one-sorted approach to plane geometry. The traditional approach to this subject is based on two-sorted structures whose sorts are a set $P$ of points and a set $L$ of lines, with a binary incidence relation $I$ from $P$ to $L$.
When $a I\ell$ the point $a $ is incident with the line $\ell$. If we allow this to also be written $\ell Ia $, then $I$ is extended to a symmetric relation on $P\cup L$. Then as  $a  I\ell Ia $ we get that $a  I^2a $, where $I^2$ is the relational composition of $I$ with itself. In a projective plane, any two points $a,b$ have a line $\ell$ passing through both, so $a  I\ell Ib $ and $aI^2b$. Each point is $I^2$-related to all points and not to any lines, and dually each line is $I^2$-related to all lines and not to any points.

The central idea of  \cite{pled:some81} was that the two-sorted theory ``can be made one-sorted by keeping careful account of whether the incidence relation is iterated an even or odd number of times''  \cite[p.1]{pled:some81}. This was implemented by imposing first-order conditions on an abstract structure $(X,I)$ with $I\sub X\times X$  that ensure that the composition $I^2$ is an equivalence relation on the set $X$ that has at most two equivalence classes. When it has exactly two, one of them can be chosen as $P$ and the other as $L$ to obtain the structure of a two-sorted projective plane. When there is only one $I^2$-equivalence class, it exhibits the structure of an elliptic plane, a type of geometry that was classically modelled by the surface of a Euclidean sphere with antipodal points identified.

We will refer to the structures $(X,I)$ defined in  \cite{pled:some81} for which this analysis holds as \emph{one-sorted  planes}. Those with two $I^2$-equivalence classes are called  \emph{projective} in \cite{pled:some81}. They include the structures $(P\cup L,I)$ where $I$ is the symmetric incidence relation of a traditional two-sorted projective plane based on $P$ and $L$. The ones with a single $I^2$-equivalance class are called \emph{elliptic}.

Now a structure of the form $\F=(X,I)$ is known in modal logic as a \emph{Kripke frame}, although Kripke called it a \emph{model structure} \cite{krip:sema63a}. A model on $\F$ for the language of modal logic is given by a truth relation specifying which formulas are true at which points of $X$, with $\bo\ph$ being true at point $a$ when $\ph$ is true at all points of $\{b\in X:aIb\}$. A formula is \emph{valid} in the frame $\F$  if it is true at every point in every model on $\F$.

Some of the first-order conditions defining a plane  $\F$ are equivalent to the validity in $\F$ of certain modal formulas. This led Pledger to study a logic, which he called 12g, which can be described as the extension of the modal logic KDB by the axiom
\begin{equation*}
\bo\bo\ph\to\bo\bo\bo\bo\ph,
\end{equation*}
whose validity is equivalent to $I^2$ being transitive.
Here KDB is the smallest normal modal logic to contain the deontic axiom
\begin{align*}
\mathrm{D:}\quad\bo\ph\to\di\ph,
\end{align*}
which corresponds to $I$ being serial (each element is $I$-related to something),
and the Brouwerian axiom
\begin{align*}
\mathrm{B:}\quad\di\bo\ph\to\ph,
\end{align*}
which corresponds to symmetry of $I$.
%
%\begin{enumerate}
%\item[($\mathrm{B}$)]
%$\quad\di\bo\ph\to\ph$
%\end{enumerate}
%
12g is sound for validity in all one-sorted  planes, including those derived from two-sorted projective planes as well as the elliptic ones.

The name `12g' relates to the fact that this logic has 12 non-equivalent proper affirmative modalities \cite[(37)]{pled:some81}, these being the finite sequences of the symbols $\bo$ and $\di$.\footnote{The naming system used in \cite{pled:some81} labels each logic  by the number of its proper affirmative modalities and a distinguishing letter.}  Any such sequence is reducible to one of 12 cases. It is shown in \cite[pp.12--15]{pled:some81} that there are nine different logics  that arise by adding to 12g an axiom of the form $A\ph\to B\ph$, where $A$ and $B$ are affirmative modalities. One of these logics, called 8f, can be obtained by adding to 12g  the axiom
\begin{equation*}
\mathrm{T}^3:\quad \bo\bo\bo\ph\to\ph
\end{equation*}
(the reason for  the name T$^3$ will be given in Section \ref{sec:strong}). 

8f is sound for validity in one-sorted elliptic planes. This can be illustrated in the classical model of the real elliptic plane, whose lines correspond to the great circles on a Euclidean sphere (with antipodal points identified). Each point has an associated line called its \emph{polar}\,: if the point is taken as the north pole then its polar is the equator.  Each line is the polar of a unique point, called the \emph{pole} of the line. The correspondence between poles and polars gives a bijection between the set of points and the set of lines. Representing each line by its pole allows the geometry to be determined by one sort of object (points), using the  binary relation $aIb$ that holds between points $a$ and $b$ when $b$ lies on the polar of $a$, which implies that $a$ lies on the polar of $b$. In this situation $a$ and $b$ are said to be \emph{conjugate}.
 If $c$ is the pole of the line joining two points $a$ and $b$, then $aIcIb$, showing that $aI^2b$. Hence there is one $I^2$-equivalence class.
Raphael Robinson\cite{robi:bina59} showed that the conjugacy relation can serve as the sole primitive from which to construct the geometry of the real elliptic plane. The line corresponding to a point $a$ is recoverable as the set $\{b:aIb\}$ of points conjugate to $a$.
Marek Kordos \cite{kord:elli73} has given an axiomatisation of plane elliptic geometry over formally real Pythagorean fields as a theory of a single binary relation.

Now a feature of the real elliptic plane is the notion of a \emph{self-polar triangle}, a triple $a,b,c$ of points such that each is the pole of the line joining the other two,  implying that  $aIbIcIa$. Hence $aI^3a$ where $I^3$ is the 3-fold composition of $I$ with itself. Each point $a$ is a vertex of a self-polar triangle so $I^3$ is reflexive, a property that is equivalent to the validity of  T$^3$. In Pledger's theory, reflexivity of $I^3$ ensures that a one-sorted plane $(X,I)$ has only one $I^2$-equivalence class and so is elliptic in that sense.

Some of the first-order conditions defining a plane  are not equivalent to the validity  of any modal formula. This was shown in Chapter VII of  \cite{pled:some81} by exhibiting counter-examples in which the condition in question fails to be preserved by a surjective \emph{bounded morphism}, a type of map between frames that preserves validity of modal formulas. Most notably this failure applies to the property of \emph{uniqueness} of the line joining two given points, and of the point at which two lines meet.

At the end of  \cite{pled:some81} the problem is raised of axiomatising the logic comprising all the modal formulas that are valid in all one-sorted projective and elliptic planes. As mentioned above this logic includes the system 12g, since 12g is \emph{sound} for validity in this class of planes. 
We will prove here that 12g is also \emph{complete} for validity in this class: any formula valid in these planes is a theorem of 12g,  so the sought for logic is just 12g itself. The proof makes use of a construction developed by  Yde Venema \cite{vene:point99} a couple of decades after Pledger's work.

In the 1990's a number of modal logicians became interested in the use of incidence geometries as models of various modal languages (see the references of \cite{vene:point99} and the review article \cite{balb:logi07} for information about this). Venema's contribution was to axiomatise the logic of two-sorted projective planes in a two-sorted modal language. In so doing he gave a step-by-step procedure for constructing a two-sorted projective plane as a bounded morphic preimage of a \emph{quasi-plane}, a structure satisfying the existence properties of projective planes, but not their uniqueness properties.

Now the completeness of 12g can also be expressed by saying that every 12g-consistent formula is satisfiable, i.e.\ true at some point,  in a model on a projective or elliptic plane. We will prove the stronger property that every 12g-consistent \emph{set} of formulas is  satisfiable in a model on the one-sorted plane  $(P\cup L,I)$ derived from some two-sorted projective plane.  We define a one-sorted notion of quasi-plane, and show that if  $\F$ is one of these and is connected, then it is a bounded morphic image of a two-sorted quasi-plane (Theorem \ref{conpreimage}). Combined with the construction from \cite{vene:point99}, we then infer that $\F$ is a bounded morphic image of the one-sorted plane derived from a two-sorted projective plane (Corollary \ref{quasipreimage}).

We then go on to show that 8f axiomatises the modal logic of elliptic one-sorted planes. The construction from \cite{vene:point99} does not seem applicable for this, since it leads to projective planes. Instead we apply the step by step \emph{method} 
of \cite{vene:point99} to directly construct a one-sorted elliptic plane as a bounded morphic preimage of a one-sorted quasi-plane (Theorem \ref{ellipticpreimage}).

In the final section, the filtration method is used to show that 12f and 8g  have the finite model property, i.e.\ are characterised by validity in \emph{finite} structures, which implies that they  are decidable.

\section{2-planes}

A \emph{two-sorted frame}, or \emph{$2$-frame}, is a structure $\F=(P,L,I)$ with $P\cap L=\emptyset$ and $I\sub P\times L$.
Elements of $P$ and $L$ are called \emph{points} and \emph{lines}, respectively, and $I$ is the \emph{incidence relation}.
When $aIb$ we may say that $a$ and $b$ are incident, that $a$ lies on $b$, that $b$ passes through $a$, etc.
Given points are said to be \emph{collinear} if there exists a line passing through all of them. A \emph{quadrangle} is a sequence of four points, no three of which are collinear. $\F$ is a \emph{projective plane} if it satisfies the following axioms:
\begin{enumerate}[\rm P1:]
\item
 any two distinct points have exactly one line passing through them both.
\item
any two distinct lines have  exactly one point lying on them both.
\item
there exists a quadrangle.
\end{enumerate}
Here (P3) is a \emph{non-degeneracy} condition. It is equivalent to requiring that the plane contains a triangle (three non-collinear points) and every line has at least three points.
We sometimes call a projective plane a \emph{projective $2$-plane} to emphasise the number of sorts involved.

%A 2-frame satisfying (P1) and (P2) but not (P3) is sometimes called a \emph{degenerate plane}.

A \emph{quasi-plane} \cite[Def.~3.1]{vene:point99} is a 2-frame satisfying:
\begin{enumerate}[\rm Q1:]
\item
 any two  points have at least one line passing through them both.
\item
any two  lines have  at least one point lying on them both.
\end{enumerate}
The two points and two lines hypothesized in these statements are not required to be distinct. Thus in a quasi-plane, any point is incident with at least one line, and any line with at least one point. 

If $\F=(P,L,I)$ and $\F'=(P',L',I')$ are 2-frames, then a \emph{homomorphism} from $\F$ to $\F'$ is given by a function $\th:P\cup L\to P'\cup L'$ with $\th(P)\sub P'$ and $\th(L)\sub L'$ that satisfies the `Forth' condition
\begin{enumerate}[F:]
\item
$aI b$ implies $\th(a)I'\th(b)$.
\end{enumerate}
A homomorphism $f$ is a \emph{bounded morphism} if  the following `Back' conditions hold for any
$a\in P$ and  $b'\in L'$ in B1, and any $a'\in P'$ and $b\in L$ in B2. 
\begin{enumerate}[B1:]
\item
If $\th(a)I' b'$, then there exists $b$ such that $aIb$ and $\th(b)=b'$.
\item 
If $a'I' \th(b)$, then there exists $a$ such that $aIb$ and  $\th(a)=a'$.
\end{enumerate}
If there exists a bounded morphism from $\F$ to $\F'$ that is surjective, then we say that $\F'$ is a \emph{bounded morphic image} of $\F'$.

\begin{theorem}  \emph{\cite[Theorem 3.2]{vene:point99}}    \label{venema}
Every quasi-plane is a bounded morphic image of a projective plane.
\qed
\end{theorem}

Venema's technique for proving  this result  starts with a quasi-plane $\F'$ and a diagram consisting of a point $a$ and a line $b$ of $\F'$ that are incident and, in a step by step manner, adds points and lines to the diagram in order to fulfil the existential assertions made by P3, Q1, Q2, B1 and B2. This is done in such a way that at each step the diagram constructed so far has a homomorphism to $\F'$ and preserves the uniqueness requirements of P1 and P2. At the end of the construction, the diagram has become a projective plane with a bounded morphism onto $\F'$.

The technique was applied to affine planes by Hodkinson and Hussain in \cite{hodk:moda08}. In   Section \ref{sec:elliptic} we will apply it to one-sorted structures.

\section{1-planes}  \label{sec:1planes}

A \emph{one-sorted frame}, or \emph{$1$-frame} is a structure $\F=(X,I)$ with $I\sub X\times X$. We continue to refer to $I$ as an \emph{incidence} relation. $I$ is \emph{serial} if its domain is $X$, i.e.\  every element $a$ of $X$ is incident ($aIb$) with some element $b$. We sometimes attribute a property of $I$ to $\F$, saying that $\F$ is serial when $I$ is etc.

Viewing $\F$ as a graph, we say that a sequence in $\F$  of the form
 $a_0Ia_1I\cdots Ia_{n}$ is an \emph{$I$-path from $a_0$ to $a_n$ of length $n\geq 1$} (the $a_i$'s need not be distinct).
% \end{document}
 Let $I^n$ be the $n$-fold relational composition of $I$ with itself.  Then $aI^{n}b$ iff there exists an $I$-path of length $n$ in $\F$ from $a$ to $b$. We also let $I^0$ be the identity relation on $X$.

A \emph{$4$-cycle} is a 
path $aIbIcIdIa$ of length 4 from an element to itself.
It is \emph{proper} if $a\ne c$ and $b\ne d$.

A \emph{one-sorted plane}, or \emph{$1$-plane}, is a 1-frame that satisfies the three conditions
\begin{enumerate}[\rm O1:]
\item
For all $a$ and $b$, $aI^4b$ implies $aI^2b$.
\item
For all $a$ and $b$, $aI^2b$ or $aI^3b$.
\item
There are no proper 4-cycles, i.e.\  $aIbIcIdIa$ implies $a=c$ or $b=d$.
\end{enumerate}
It is called \emph{non-degenerate} if it also satisfies
\begin{enumerate}[\rm O4:]
\item
There exist $a,b,c,d,e,f$ with $aIbIcId$ and $eIf$, but 
$a$ is not incident with $d$, and none of $a$, $b$, $c$, $d$ is incident with $e$ or $f$.
\end{enumerate}

Let $ I_+=I\cup\{(b,a):aIb\}$,  the smallest symmetric relation including $I$. $I$ itself is symmetric iff $I_+=I$.
If $\F=(P,L,I)$ is a 2-frame, let  $\F_+$ be the 1-frame $(P\cup L, I_+)$.

\begin{theorem}  \emph{\cite[Chapter I]{pled:some81}}    \label{proj1plane}
If $\F$ is a projective $2$-plane, then $\F_+$ is a non-degenerate $1$-plane in which $P$ and $L$ are distinct equivalence classes under $I_+^2$, and the only such classes.
\end{theorem}

\begin{proof}
Conditions O1--O4 are axioms (1)--(4) of \cite{pled:some81}. We reprise the arguments given there for the validity in
 $\F_+$ of these conditions. For O1, if $aI_+{}^nb$ with $n$ even, then $a$ and $b$ are both points or both lines in $\F$, hence $a I_+{}^2b$
as noted in Section \ref{sec:diss}.

For O2, if $a$ and $b$ are of the same sort in $\F$, then $aI_+{}^2b$ as just observed. If $a$ is a point and $b$ is a line, take any line $c$  through $a$ and let point $d$ lie on $b$ and $c$: then $aI_+cI_+dI_+b$, hence  $aI_+{}^3b$. If  $b$ is a point and $a$ is a line, then  $bI_+{}^3a$ by the previous sentence, hence  $aI_+{}^3b$.

O3 captures the uniqueness expressed in P1 and P2. Let $aI_+bI_+cI_+dI_+a$.
If $a$ is a point, then so is $c$ while $b$ and $d$ are lines passing through both points, hence if $a\ne c$ then $b=d$ by P1.
If $a$ is a line, then so is $c$ while $b$ and $d$ are points on both lines, so again if $a\ne c$ then $b=d$ by P2. 

O4 captures the non-degeneracy expressed by P3, which ensures that $\F$ has some quadrangle $b,d,e,g$.  Figure \ref{figone}  shows how to then obtain lines $a,c,f$ so that $a,b,c,d,e,f$ fulfils O4 in $\F_+$  \cite[Figure 6]{pled:some81}.
%%%%%%%%%%%%%%%%%%%%
\begin{figure}[h]       
\begin{center}{   
 \setlength{\unitlength}{2pt}
 
\begin{picture}(40,40)  
%\thicklines   
 \put(-2,-6){\line(1,3){14}}
 \put(-5,0){\line(1,0){47}}
 %\put(0,0){\line(5,4){25}} 
%\put(30,0){\line(-2,3){20}} 
%\put(30,0){\line(-1,4){12.9}} 
\put(7,32){\line(3,-2){38}} 
 
\put(0,0){\circle*{1.5}} 
\put(30,0){\circle*{1.5}} 
\put(10,30){\circle*{1.5}} 
\put(25,20){\circle*{1.5}} 

\put(1,17){$a$}
\put(-3,2){$b$}
\put(15,2){$c$} 
\put(29,-5){$d$} 
\put(26,21){$e$} 
\put(40,13){$f$} 
\put(12,31){$g$} 

\end{picture}}
\end{center}
\caption{non-degeneracy}
\label{figone}
\end{figure}
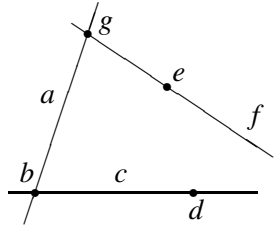

That proves  $\F_+$ is a non-degenerate $1$-plane.  Now given distinct points $a,b\in P$,  there is a line $c\in L$ passing through both, hence $aI_+cI_+b$  so $aI_+^2b$. Also any point $a$ is incident with some line $c$, so that $aI_+cI_+a$ and hence $aI_+^2a$. Thus any point is $I^2$-related to all points, but not to any lines because $I$ only connects elements of different type, so $I^2$ only connects elements of the same type. Dually,  any line is $I^2$-related to all lines, but not to any points.  Since $P$ and $L$ are disjoint they form a partition of $P\cup L$ with each being an $I^2$-equivalence class.
\end{proof}
%%%%%%%%%%%%%%%%%%%%%%%%%%

A 1-frame $\F=(X,I)$ is \emph{connected} if any two of its elements have an $I_+$-path from one to the other. If $I$ is symmetric, this is an $I$-path. It follows from O2 that a 1-plane is connected.

It was shown in \cite[Chapter 1]{pled:some81} that in any 1-plane, $I$ is serial and symmetric,  $I^2$ is an equivalence relation with at most two equivalence classes, and when there are two such classes then by taking one as $P$ and the other as $L$ and restricting $I$ to $P\times L$ we obtain a 2-frame satisfying P1and P2. If the 1-plane is non-degenerate, then P3 holds as well, so it is a projective plane. When there is one $I^2$-equivalence class, by regarding each element as both a point and a line we can reach the same conclusion.

It was then observed on \cite[p.7]{pled:some81} that parts of this analysis depend only on weaker assumptions. To explain this we define a \emph{quasi-$1$-plane}  to be any 1-frame that is \emph{serial} and \emph{symmetric} and \emph{satisfies} O1. Thus any 1-plane is a quasi-$1$-plane.
The quasi-planes of the previous section may  be called \emph{quasi-2-planes} to distinguish them from the one-sorted notion just defined.

\begin{theorem}   \label{quasi1props}
Let $\F=(X,I)$ be a quasi-$1$-plane. Then:
\begin{enumerate}[\rm(1)]
\item 
$I^2$ is an equivalence relation.
\item
$aIb$ implies $aI^3b$.
\item
$\F$ is connected iff it satisfies O2.
\item
If $\F$ is connected, then for any element $a$, the subsets $I^2(a)=\{x:aI^2x\}$ and $I^3(a)=\{x:aI^3x\}$ are the only $I^2$-equivalence classes.
\item
If $\F$ is connected, then for any element $a$, if $xIy$ then one of $x$ and $y$ is in $I^2(a)$ and the other is in $I^3(a)$.
\end{enumerate}
\end{theorem}

\begin{proof}
\begin{enumerate}[\rm(1)]
\item
Reflexivity of $I^2$: for any $a$ there a $b$ with $aIb$ by seriality of $I$, so then $aIbIa$ by symmetry of $I$, hence $aI^2a$. Symmetry of $I^2$ follows from symmetry of $I$. Transitivity of $I^2$ is just what O1 asserts.
\item
If $aIb$, then $aIbI^2b$ by (1), giving $aI^3b$.
\item
O2 implies that any two points have an $I$-path between them of length 2 or 3, so $\F$ is connected. For the converse, take $a,b\in X$. If $a=b$, then $aI^2b$ by (1). If $a\ne b$, then assuming $\F$ is connected, since $I$ is symmetric we have $aI^n b$ for some $n\geq 1$. If $n=1$ then $aI^3b$ by (2). If $n\geq 4$, then since O1 ensures that any path  of length 4 between two points can be replaced by one of length 2, we get that $aI^{n-2} b$, then $aI^{n-4} b$ etc., leading to either  $aI^2b$ or $aI^3b$. Hence O2 holds.
\item
(See \cite[proof of (13)]{pled:some81}.)
Assume $\F$ is connected, and hence satisfies O2 by (3).  Fix any $a\in X$. As in (1), there is some $b$ with $aIbIa$. If $aI^3x$, then $bIaI^3x$, so $bI^4x$ and hence $bI^2x$ by O1. Conversely, if $bI^2x$ then $aIbI^2x$ and so $aI^3x$. This shows that $I^3(a)$ is the $I^2$-equivalence class $I^2(b)$ of $b$. But by O2, for any $x$, either $aI^2x$ or $aI^3x$, so the $I^2$-equivalence classes of $a$ and $b$ are all that there are.
\item
Let $xIy$. If $x\in I^2(a)$, then $aI^2xIy$, so $y\in I^3(a)$. But if $x\notin I^2(a)$, then $x\in I^3(a)$ by part (4), so $aI^3xIy$, hence $aI^4y$, then $y\in I^2(a)$ by O1.
\qedhere
\end{enumerate}
\end{proof}

By part (4) of this result, all connected quasi-1-planes, which includes all the  1-planes, have at most two $I^2$-equivalence classes.  Those with two classes will called \emph{projective} quasi-1-planes, while those with one class are \emph{elliptic}. So Theorem \ref{proj1plane} states that if $\F$ is a projective $2$-plane, then $\F_+$ is a non-degenerate projective $1$-plane.

Now a symmetric 1-frame can be viewed as an undirected graph with an edge joining vertices $a$ and $b$ just when $aIb$ (hence $bIa$). Figure \ref{twographs} 
%%%%%%%%%%%%%%%%%%%%%
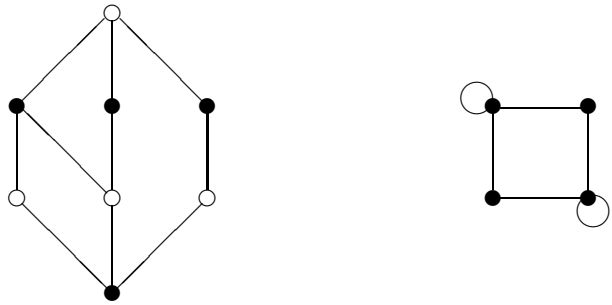
\begin{figure}[h]  
\begin{center}{   
 \setlength{\unitlength}{2pt}

\begin{picture}(130,40)            
%\thicklines   

\put(0,0){\circle{3}}
\put(18,0){\circle{3}}
\put(36,0){\circle{3}}

\put(0,17.5){\circle*{3}}
\put(18,17.5){\circle*{3}}
\put(36,17.5){\circle*{3}}

\put(18,35){\circle{3}}
\put(18,-18){\circle*{3}}

 \put(1,-1){\line(1,-1){16}}
 \put(17,1){\line(-1,1){16}}
 
 \put(0,1.5){\line(0,1){15}}
% \put(1.5,17.5){\line(1,0){15}}
  \put(35,-1){\line(-1,-1){16}} 
   \put(36,1.5){\line(0,1){15}} 
    \put(18,16.5){\line(0,-1){15}}
  \put(18,18.5){\line(0,1){15}}    
  \put(18,-1.5){\line(0,-1){15}} 
  
  \put(0,17.5){\line(1,1){16.5}}
   \put(36,17.5){\line(-1,1){16.5}}
   
%%%%%  2nd picture  %%%%%%%
  
  \put(90,0){\circle*{3}}
  \put(90,17.5){\circle*{3}}
 \put(108,0){\circle*{3}}
  \put(108,17.5){\circle*{3}}
    
\put(87,19){\circle{6}}
  \put(109,-2.5){\circle{6}}
  
   \put(90,1.5){\line(0,1){15}}
   \put(91.5,0){\line(1,0){15}}
   
   \put(91.5,17){\line(1,0){15}}
  \put(108,0.5){\line(0,1){16}}   
 
 \end{picture}}

\end{center}

\bigskip\bigskip\bigskip
\caption{two quasi-1-planes, not 1-planes}
\label{twographs}
\end{figure}
%%%%%%%%%%%%%%%%%%
shows two such graphs,  which are Figures 105 and 107 from \cite{pled:some81}. The left one is a projective quasi-1-plane with the black and white circles displaying the two $I^2$-equivalence classes. This graph also satisfies the non-degeneracy condition O4. The right graph is an elliptic quasi-1-plane, with the two looped edges displaying $I$-reflexive vertices. Both graphs contain proper 4-cycles violating O3, so neither is a 1-plane.

We now extend the notion of bounded morphism to variously-sorted structures. First, a bounded morphism from a 1-frame $\F=(X,I)$ to a 1-frame $\F'=(X',I')$ is given by a function $\th:X\to X'$ that is a homomorphism, i.e.\ satisfies the forth condition F for all $a,b\in X$, and satisfies the back condition B1 for all $a\in X$ and $b\in X'$.
This is the standard notion of bounded morphism for 1-frames in modal logic.
It is a routine fact that the functional composition of two bounded morphisms of this type is also a bounded morphism.

Next we define a bounded morphism from a 2-frame $\F=(P,L,I)$ to a 1-frame $\F'=(X',I')$ as being given by a function $\th:P\cup L\to X'$ that satisfies the homomorphism condition F for all $a\in P$ and $b\in L$, the condition B1 for all $a\in P$ and $b\in X'$; and the condition B2 for all  $a'\in X'$ and $b\in L$.

\begin{lemma}  \label{bmorphs}
\begin{enumerate}[\rm(1)]
\item 
If $\th$ is a bounded morphism from a 2-frame $\F=(P,L,I)$ to a 2-frame $\F'=(P',L',I')$, and $\tau$ a  bounded morphism from  $\F'$ to a 1-frame $\F''=(X'',I'')$, then the composition $\tau\circ\th$ is a bounded morphism from $\F$ to $\F''$.
\item
If $\th$ is a bounded morphism from a 2-frame $\F=(P,L,I)$ to a 1-frame $\F'=(X',I')$, and $I'$ is symmetric, then $\th$ is also a bounded morphism from the 1-frame $\F_+=(P\cup L, I_+)$ to $\F'$.
\end{enumerate}
\end{lemma}

\begin{proof}
\begin{enumerate}[\rm(1)]
\item
 $\tau\circ\th$ is a homomorphism because $\th$ and $\tau$ are: if $aIb$ then $\th(a)I'\th(b)$ by F for $\th$ and so by F for 
 $\tau$,   $\tau(\th(a))I''\tau(\th(b))$ , i.e.\ $(\tau\circ \th)(a)I''(\tau\circ \th)(b)$. 
 
 For B1, if  $\tau(\th(a))I'' b''$, then by B1 for $\tau$ there exists $b'\in L'$ with $\th(a)I' b'$ and $\tau(b')=b''$. Hence by B1 for $\th$, there exists $b\in L$ with $aIb$ and $\th(b)=b'$, hence $(\tau\circ \th)(b)=b''$.
 The proof of B2 for $\tau\circ \th$ is similar,
\item
F: let $aI_+b$ in $\F_+$. If $aIb$, then $\th(a)I'\th(b)$ by F for $\th:\F\to\F'$.
Otherwise we have $bIa$, hence $\th(b)I'\th(a)$ likewise, and so again $\th(a)I'\th(b)$ as $I'$ is symmetric.
This shows that F holds for $\th:\F_+\to\F'$.

B1: let $\th(a)I'b'$ in $\F'$.  If $a\in P$, then by B1 for $\th:\F\to\F'$ there exists $b\in L$ with $aIb$, hence $aI_+ b$, and $\th(b)=b'$. Otherwise we have $a\in L$.
But $b'I'\th(a)$ by symmetry, so by B2 for $\th:\F\to\F'$ there exists $b\in P$ with $bIa$, hence $aI_+ b$, and $\th(b)=b'$.
This shows B1 holds for $\th:\F_+\to\F'$.
\qedhere
\end{enumerate}
\end{proof}

\begin{theorem} \label{conpreimage}
Every connected  quasi-1-plane is a bounded morphic image of a quasi-2-plane.
\end{theorem}

\begin{proof}
Let $\F=(X,I)$ be a connected  quasi-1-plane.
By Theorem \ref{quasi1props}(4), $\F$ has at most two $I^2$-equivalence classes. 
Take first the case that it has exactly two.
Take an element $a$ and let $P=\{x:aI^2x\}$ and $L=\{x:aI^3x\}$.
 Theorem \ref{quasi1props}(4) says these are the two $I^2$-equivalence classes, 
 so they are disjoint and partition $X$. Form the 2-frame
    $\F'=(P,L,I')$, where $I'$ is the restriction of $I$ to $P\times L$.
    Let $\th:P\cup L\to X$ be the identity function.
    
Since $I'\sub I$, $\th$ is a homomorphism from $\F'$ onto $\F$. To see it satisfies B1, let $b\in P$ and $\th(b)Ic$. Then $aI^2bIc$, so $aI^3c$ and hence $c\in L$ with $bI'c$. As $\th(c)=c$, this proves B1.   For B2, let $c\in L$ and $bI\th(c)$. Then $aI^3 c$ and $bIc$, hence  $cIb$ by symmetry, so $aI^4b$. But then $aI^2b$ by O1, so $b\in P$, with $bI'c$ and $\th(b)=b$.

We have now shown that $\th$ is a surjective bounded morphism from $\F'$ onto $\F$. It remains to show that $\F'$ is a a quasi-2-plane. For Q1, take two points $b,d\in P$. Then $bI^2d$ as they are in the same equivalence class, so $bIcId$ for some $d$. As $b\in P$ and $bIc$ we get $c\in L$ as above, so $c$ is a line. As $bIc$ and $dIc$, this line passes through $b$ and $d$. The proof of Q2 is similar: any two lines $b,d\in L$ have $bIcId$, for some $c\in P$ with $c$ lying on $b$ and $d$.

That concludes the first case. The alternative is that $\F$ has only one $I^2$-equivalence class. Then $aI^2b$ for all $a,b\in X$. Take two disjoint copies $P_0=\{a_0:a\in X\}$ and $L_1=\{a_1:a\in X\}$ of $X$ and let
  $\F'=(P_0,L_1,I')$, where $a_0I'b_1$ iff $aIb$. Let $\th:P_0\cup L_1\to X$ be the projection: $\th(a_0)=\th(a_1)=a$.
  
  The definition of $I'$ ensures that $\th$ is a homomorphism from $\F'$ onto $\F$. For B1, if $a_0\in P_0$ and $\th(a_0)Ib$, then $aIb$ and so $a_0I'b_1$ with $\th(b_1)=b$. The proof of B2 is similar. Thus $\th$ is a  bounded morphism. 
  
 It remains to show that $\F'$ is a a quasi-2-plane in this case. But given points $a_0,b_0\in P_0$ in $\F'$ we have $aI^2b$ in $\F$, so there exists $c$ with $aIcIb$, hence $bIc$, so that $a_0I'c_1$ and $b_0I'c_1$, making $c_1$ a line in $\F'$ that passes through $a_0$ and $b_0$. That proves Q1. The proof of Q2 is similar:  any two lines $a_1,b_1$ of $\F'$ have  $aIcIb$ for some $c$, and $c_0$ is a point in $\F'$  lying on  both lines.
\end{proof}

\begin{corollary}  \label{quasipreimage}
Every connected  quasi-1-plane is a bounded morphic image of a non-degenerate projective 1-plane of the form $\G_+$ where $\G$ is a projective 2-plane.
\end{corollary}
\begin{proof}
Let $\F=(X,I)$ be a connected  quasi-1-plane. By the Theorem there is a quasi-2-plane $\F'$ and a surjective bounded morphism from $\F'$ to $\F$. By Theorem \ref{venema} there is a projective 2-plane $\G$ and a surjective bounded morphism  from $\G$ to $\F'$. By Lemma \ref{bmorphs}(1) these compose to give a surjective bounded morphism $\G\to\F$, which by
Lemma \ref{bmorphs}(2) also acts as a surjective bounded morphism $\G_+\to\F$ of 1-frames.
By Theorem \ref{proj1plane}, $\G_+$ is a non-degenerate projective  1-plane.
\end{proof}

 \section{The elliptic case}  \label{sec:elliptic}

A quasi-1-plane is elliptic when it has only one $I^2$-equivalence class, a property expressed simply by
\begin{enumerate}
\item[O5:]
For all $a$ and $b$,  $aI^2b$.
\end{enumerate}
This condition implies O1 and O2 (equivalently connectedness) as well as seriality of $I$, so we can characterise an elliptic quasi-1-plane as a  1-frame satisfying O5 and symmetry, and an elliptic 1-plane as a 1-frame satisfying O5 and O3.

\begin{lemma}
Any $1$-frame satisfying O5 and O3 is symmetric.
\end{lemma}

\begin{proof}
Let $aIb$. By seriality there exists $x$ with $bIx$. By O5 $xI^2a$, so there is a $y$ with $xIyIa$. Now we have a 4-cycle $aIbIxIyIa$. By O4 either  $a=x$ or $b=y$. In either case $bIa$.
\end{proof}
%Thus an elliptic 1-plane is an elliptic quasi-1-plane. 
Pledger suggested that the theory based on O5 and O3 ``can be regarded as elliptic plane geometry in a general sense which permits degenerate planes'' \cite[p.18]{pled:some81}. He gave an analysis of degenerate elliptic planes, showing that they can be excluded by the following shorter substitute for O4 (see  \cite[(79)]{pled:some81}):
\begin{enumerate}[O4$'$:]
\item
There exist $a,b,c,d$ such that $aIbIc$ and $a\ne c$ and $b\ne d$, and neither $aIc$ nor $bId$.
\end{enumerate}
The elliptic quasi-1-plane depicted on the right of Figure \ref{twographs} satisfies O4$'$.

 The projective 1-plane $\F_0$ depicted in Figure \ref{F0}
%%%%%%%%%%%%%%%%%%%%%
\begin{figure}[h]  
\begin{center}{   
 \setlength{\unitlength}{2pt}
 
\begin{picture}(45,20)  
%\thicklines   

\put(0,0){\circle*{3}}              %{\Large$\circ$}\
 \put(0,0){\line(1,0){13.5}}
 
 \put(15,0){\circle{3}}
 \put(16.5,0){\line(1,0){15}}
 
 \put(30,0){\circle*{3}}
 \put(30,0){\line(1,0){13.5}}
 
 \put(45,0){\circle{3}}
% \put(45,0){\line(1,0){15}}

\put(-3,5){$v_0$}
\put(12,5){$v_1$}
\put(27,5){$v_2$}
\put(42,5){$v_3$}
 
 \end{picture}}

\end{center}
\caption{\ $\F_0$}
\label{F0}
\end{figure}
%%%%%%%%%%%%%%%%%%
has four vertices, with an edge joining $v_i$ to $v_{i+1}$ for each $i<3$, and no other edges. It satisfies O4$'$. Any given symmetric frame will satisfy O4$'$ if it contains (a copy of) $\F_0$ as a \emph{full} subgraph, i.e.\ the frame  has no edges between vertices of $\F_0$ other than the three edges already belonging to $\F_0$.

We want to extend Corollary \ref{quasipreimage} to the elliptic case, to show that any elliptic quasi-1-plane is a bounded morphic image of an elliptic 1-plane that is non-degenerate in the sense of O4$'$. But the first step in the proof of that corollary destroys the elliptic property O5. So instead we will directly build a 1-plane out of an elliptic quasi-1-plane by adapting the step by step method of \cite{vene:point99}.

The uniqueness property O3 can be given a more evocative formulation when $I$ is symmetric. If $aIb$ we may say that $a$ and $b$ are \emph{neighbours}.
If $aIbIc$ then $b$ is a \emph{common neighbour of $a$ and $c$}. If also $d$ is a common neighbour of $a$ and $c$, then (using symmetry) we have the 4-cycle $aIbIcIdIa$, so if $a\ne c$, then O3 enforces $b=d$. We see that in the presence of symmetry, O3 is equivalent to the requirement that  any two distinct elements have \emph{at most one} common neighbour.
Combining that with O5 we get that in an elliptic 1-plane any two distinct elements have \emph{exactly one} common neighbour.

\begin{theorem} \label{ellipticpreimage}
Every elliptic  quasi-1-plane is a bounded morphic image of a non-degenerate elliptic 1-plane.
\end{theorem}
\begin{proof}
This will occupy the rest of the current section.
Fix an elliptic  quasi-1-plane $\F'=(X',I')$.
Define a \emph{network} to be a structure $\N=(X,I,\th)$ comprising a \emph{symmetric} 1-frame $(X,I)$, which we may 
denote   $\F_{\N}$,
and a function  $\th:X\to X'$, which we may denote $\th_{\N}$. We say that another network $\N^+=(X^+,I^+,\th^+)$ \emph{extends} $\N$ if $X\sub X^+$, $I$ is the restriction of $I^+$ to $X\times X$, and $\th$ is the restriction of $\th^+$ to $X$.

A network $\N$ is called \emph{coherent} if it satisfies the following conditions
\begin{enumerate}[C1:]
\item 
$\th_{\N}:\F_{\N}\to\F'$ is a homomorphism, i.e.\ satisfies the forth condition F.
\item
distinct elements of $X$ have at most one common neighbour.
\item
$\F_0$ is a full subgraph of $\F_{\N}$.
\end{enumerate}
For instance, we obtain a coherent network $\N_0=(\F_0,\th_0)$ by noting that since $\F'$ is serial it contains a 3-path $a_0Ia_1Ia_2Ia_3$, and putting $\th_0(v_i)=a_i$ for all $i\leq 3$. We want to build $\N_0$ up to a coherent network $\N$  that satisfies O5 and has $\th_{\N}$ as a bounded morphism. A given network may fail one or both of these requirements, and instances of such failure are called \emph{defects}. These will require \emph{repair}. There are two types of possible defect for a network $\N=(X,I,\th)$:

\begin{description}
\item[\rm B1-\emph{defect}:] 
this is a pair $(a,b')\in X\times X'$ such that $\th(a)I'b'$ but there is no $b\in X$ with $aIb$ and $\th(b)=b'$.

\item[\rm O5-\emph{defect}:]
 a pair $(a,b)\in X\times X$ with no common neighbour, i.e.\ $aI^2b$ fails.
\end{description}
We show that any defect in a coherent network can be repaired by constructing a coherent extension that contains the element described as lacking by the defect. Once a defect is repaired, it remains so in all further coherent extensions.

To repair a B1-defect $(a,b')$ in a coherent network $\N$, take a new object $b$ not in $X$ or $X'$ and extend $\N$ to $\N^+$ by putting
\begin{align*}
X^+ &=X\cup\{b\},
\\
I^+ &= I\cup\{(a,b),(b,a)\}
\\
\th^+&=\th\cup\{(b,b')\}.
\end{align*}
The definition of $I^+$ ensures that it is symmetric because $I$ is. Thus $\N^+$ is a network having $aI^+b$ and $\th^+(b)=b'$, so $(a,b')$ is no longer a B1-defect in $\N^+$. It remains to check that $\N^+$ is coherent. For C1, since $\N$ satisfies C1 we only have 
to check that $\th^+(a)I'\th^+(b)$ and $\th^+(b)I'\th^+(a)$. But this follows because $\th^+(a)=\th(a)I'b'=\th^+(b)$  and  $I'$ is symmetric. For C2, take two distinct elements of $X^+$. If one of them is $b$, then they have at most one common $\N^+$-neighbour because $b$ has only the one $\N^+$-neighbour $a$. If neither of them is $b$ then they have at most one common neighbour in $\N$ by C2 for $\N$, and they do not have $b$ as a common neighbour as $b$ has only one neighbour, so they have at most one neighbour in $\N^+$. For C3, $\F_0$ has no extra edges in $\N$ by C3 for $\N$, and $\N^+$ has only one more edge, with a vertex $b$ not in $\N$, hence not in $\F_0$, so $\F_0$ has no extra edges in $\N^+$.
Thus $\N^+$ is coherent.

To repair an O5-defect $(a,b)$ in a coherent network $\N$, observe that as $\F'$ is elliptic we have  $\th(a)(I')^2\th(b)$, so there exists a $d\in X'$ with $\th(a)I'dI'\th(b)$. 
Take a new object $c\notin X\cup X'$ and extend $\N$ to $\N^+$ by putting
\begin{align*}
X^+ &=X\cup\{c\},
\\
I^+ &= I\cup\{(a,c),(c,a), (c,b),(b,c)\}
\\
\th^+&=\th\cup\{(c,d)\}.
\end{align*}
Again the definition of $I^+$ ensures that it is symmetric because $I$ is. Thus $\N^+$ is a network having $aI^+cI^+b$, hence  $a(I^+)^2b$, so $(a,b)$ is no longer an  O5-defect in $\N^+$. We check that $\N^+$ is coherent. For C1, it suffices since $\N$ satisfies C1 to show that  $\th^+(a)I'\th^+(c)$, $\th^+(c)I'\th^+(b)$ and their inverses. But this follows because $\th^+(a)=\th(a)I'd=\th^+(c)$  and similarly  $\th^+(c)=dI'\th^+(b)$, and $I'$ is symmetric.
For C2, let $x$ and $y$ be \emph{distinct} members of $X^+$. If one of them, say $y$, is equal to $c$, then $x\in X$. If $x$ and $y=c$ have more than one common neighbour in $\N^+$, then their common neighbours must be $a$ and $b$, as those are all the neighbours $c$ has. Hence  $a$ and $b$ are neighbours of $x$. But then $aIxIb$, contradicting the fact that $(a,b)$ is an O5-defect. So in this case
$x$ and $y$ have at most one common neighbour in $\N^+$. The alternative case is that $x$ and $y$ both belong to $X$.  Now if $x$ and $y$ both belong to $\{a,b\}$, then they have no common neighbour in $\N$, as $(a,b)$ is an O5-defect, so by construction they have $c$ as their only common neighbour in $\N^+$. But if one of them, say $x$, does not belong to  $\{a,b\}$, then it does not have $c$ as a neighbour, so does not have $c$ as a common neighbour with $y$. 
Since $x$ and $y$ have at most one common neighbour in $\N$, as $\N$ satisfies C2, it follows in this case that they continue to have at most one common neighbour in $\N^+$. That completes the proof that C2 holds in $\N^+$.
Finally, for C3 we have that $\F_0$ has no extra edges in $\N$ by C3 for $\N$, and $\N^+$ has only two additional edges both with a vertex $c$ not in $\F_0$, so $\F_0$ has no extra edges in $\N^+$.

We have now shown that  any defect in a coherent network is repaired in some coherent extension. By iterating this construction sufficiently we will eventually eliminate all defects. The first step is

\begin{lemma} \label{N*}
Every coherent network $\N$ has a coherent extension $\N^+$ such that every defect of $\N$ is repaired in $\N^+$.
\end{lemma}
\begin{proof}
Let $\ka$ be the cardinal of the set of all defects of $\N$ (which is a subset of $(X\times X')\cup (X\times X)$), and let $\<\d_\mu:\mu< \ka\>$ be an indexing of these defects by the ordinals less than $\ka$. We inductively construct a sequence
$\<\N^*_\mu:\mu< \ka\>$ of coherent extensions of $\N$ such that for all $\xi<\mu$, $\N^*_\mu$ extends 
$\N^*_\xi$.

Put $\N^*_0=\N$.  For $\mu<\ka$, assuming inductively that a coherent $\N^*_\mu$ has been defined, let $\N^*_{\mu+1}$ be a coherent network extending $\N^*_\mu$ in which $\d_\mu$ is no longer a defect, as exists by the above constructions. If $\mu$ is a limit ordinal, assuming inductively that the sequence $\<\N^*_\xi:\xi< \mu\>$ of coherent networks has been defined and forms a chain under the extension relation, let
$$\textstyle
\N^*_\mu=\bigcup\nolimits_{\xi<\mu}\N^*_\xi %\{\N^*_xi:\xi<\mu\}
=(\bigcup\nolimits_{\xi<\mu}X^*_\xi,\bigcup\nolimits_{\xi<\mu}I^*_\xi,\bigcup\nolimits_{\xi<\mu}\th^*_\xi)
$$
be the union of the $\N^*_\xi$'s. It is readily checked that the union of an extension-chain of coherent networks is a coherent network extending each member of the chain. That completes the definition of the $\N^*_\mu$'s.

Now put $\N^+=\bigcup\nolimits_{\mu<\ka}\N^*_\mu$, a coherent network extending $\N^*_0=\N$. For each $\mu<\ka$, the defect  $\d_\mu$ is repaired in $\N^*_{\mu+1}$, and hence is repaired in  $\N^+$.
\end{proof}

Next we use this result to construct a countably infinite chain  $\<\N_n:n< \om\>$ that repairs all defects. $\N_0$ is the coherent network based on the frame $\F_0$, as defined earlier.  Assuming inductively that $\N_n$ has been defined as a coherent network, let $\N_{n+1}$ be the coherent extension $\N_n^+$ of $\N_n$ given by Lemma \ref{N*}. 

Now put
$\N_\om=\bigcup\nolimits_{n<\om}\N_n$. $\N_\om$ is coherent, and has no defects, since any purported defect of 
$\N_\om$ would  be a defect of  $\N_n$ for some $n<\om$, and hence be repaired in $\N_{n+1}$, therefore not a defect in $\N_\om$ after all.

Let $\N_\om=(\F_\om,\th_\om)$. Since $\N_\om$ has no O5-defects, $\F_\om$ satisfies O5. Since $\N_\om$ is coherent it satisfies C2 (distinct elements have at most one common neighbour) and hence $\F_\om$ satsifes O3. Thus $\F_\om$ is an elliptic 1-plane. By C3 for $\N_\om$, $\F_0$ is a full subgraph of $\F_\om$, so $\F_\om$ is non-degenerate. The map
$\th_\om:\F_\om\to\F'$ is a homomorphism by C1 for $\N_\om$. Since there are no B1-defects it is a bounded morphism.

To complete the proof of Theorem \ref{ellipticpreimage}, that $\F'$ is a bounded morphic image of a non-degenerate elliptic 1-plane,  it  remains only to show that $\th_\om$ is surjective. Let $c'$ be any element of  $X'$. Take any element $a$ of $\F_\om$. Then as $\F'$ is elliptic, $\th_\om(a)(I')^2c'$, so there exists $b'\in X'$ with
 $\th_\om(a)I'b'I'c'$. Then by B1 for $\th_\om$ there exists $b$ with $aI_\om b$ and $\th_\om(b)=b'$, so $\th_\om(b)I'c'$. By B1 again there then exists $c$ with $bI_\om c$ and $\th_\om(c)=c'$, which establishes that $\th_\om$ maps \emph{onto}
  $\F'$.
\end{proof}

\section {Some modal metatheory}

We review the background from propositional modal logic that will be needed. This  can be found in a number of texts,  such as \cite{blac:moda01,chag:moda97,gold:logi92}.

Modal formulas $\ph,\psi,\dots$ are constructed from a countably infinite set  $Var$ of propositional variables by the standard Boolean connectives  $\neg$, $\land$ and the unary modality $\bo$. The other Boolean connectives  $\lor$, $\to$, $\leftrightarrow$ are introduced as the usual abbreviations, and the dual modality $\di$ is defined to be $\neg\bo\neg$. 
Formulas $\bo^n\ph$ are defined by
induction on the natural number $n$ by putting $\bo^0\ph=\ph$ and $\bo^{n+1}\ph=\bo(\bo^n\ph)$.
Formulas $\di^n\ph$ are defined likewise by iterating $\di$. 

A \emph{model} $\M=(\F,V)$ for this language is given by a 1-frame $\F=(X,I)$ and a valuation function $V$ that  assigns to each variable $p\in Var$ a subset $V(p)$ of $X$ which may be thought of as the set of points at which $p$ is true. 
For each $a\in X$, let $I(a)=\{b\in X:aIb\}$ be the set of \emph{$I$-neighbours} of $a$.
The relation 
$\M,a\models\ph$ of a formula $\ph$ being \emph{true (or satisfied) at $a$ in $\M$} is defined by  induction on the formation of $\ph$
as follows:
\begin{itemize}%[\rm(1)]
\item $\M,a\models p$ iff $x\in V(p)$, \enspace  for $p\in Var$.

\item $\M,a\models\neg\varphi$ iff $\M,a\not\models\varphi$ \enspace(\text{i.e.\ not }$\M,a\models\ph$).

\item $\M,a\models\varphi\wedge\psi$ iff $\M,a\models\varphi$ and $\M,a\models\psi$.

\item $\M,a\models\bo\varphi$ iff for every $b\in I(a)$, $\M,b\models\varphi$ .
\end{itemize}
Consequently 
\begin{itemize}%[\rm(5)]
\item   \label{item:semdi}
$\M,a\models\di\ph$ iff for some $b\in I(a)$, $\M,b\models\varphi$.
\item
$\M,a\models\bo^n\ph$ iff for all $b$ such that $aI^nb$, $\M,b\models\ph$.
\item
$\M,a\models\di^n\ph$ iff for some $b$ such that $aI^nb$, $\M,b\models\ph$.
\end{itemize}
A formula $\ph$ is \emph{true in model $\M$} when it is true at all elements of $\M$,
and \emph{satisfiable in} $\M$ when it is true at some element of $\M$, i.e.\ $\neg\ph$ is not true in $\M$. A set $\Sig$ of formulas is \emph{satisfiable in} $\M$ when there is  some element of $\M$ at which all members of $\Sig$ are simultaneously true. $\Sig$ is \emph{satisfiable} in a 1-frame if it is satisfiable in some model on that frame.

A formula $\ph$ is  \emph{valid in $1$-frame $\F$} if it is true in all models on $\F$, or equivalently if $\neg\ph$ is not satisfiable in $\F$. $\ph$ is \emph{valid in a class} $\C$ of 1-frames if it is valid in all members of $\C$.

Given a bounded morphism $\th:\F'\to\F$ of 1-frames and a model $\M=(\F,V)$ on $\F$ we obtain a model $\M'=(\F',V')$ on $\F'$ by putting $V'(p)=\th^{-1}V(p)=\{a\in X':\th(a)\in V(p)\}$ for all  $p\in Var$. Then for any $a\in X'$, if $\ph$ is a variable,
\begin{equation}  \label{morphequiv}
\M',a\models\ph \enspace\text{iff}\enspace \M,\th(a)\models\ph.
\end{equation}
An induction on formation of formulas then shows that \eqref{morphequiv} holds for every $a$ and every formula $\ph$ \cite[Proposition 2.14]{blac:moda01}.

\begin{theorem} \label{bmorphsat}
If $\F$ is a bounded morphic image of $\F'$, and $\Sig$ is a set of formulas that is satisfiable in $\F$, then $\Sig$ is satisfiable in $\F'$.
\end{theorem}

\begin{proof}
Let $\th:\F'\to\F$ be a surjective bounded morphism and $\M$ a model on $\F$ such that $\Sig$ is a satisfiable at some element $b$ in $\M$. Take an element $a$ in $\F'$ such that $b=\th(a)$. Then from \eqref{morphequiv} we get that every member of $\Sig$ is true at $a$ in $\M'$.
\end{proof}
\noindent
Note that this result implies that if $\F$ is a bounded morphic image of $\F'$, then any formula $\ph$ that is valid in $\F'$ must be valid in $\F$, for if $\neg\ph$ were satisfiable in $\F$ it would be satisfiable in $\F'$.

An \emph{inner subframe} of a frame $\F=(X,I)$ is any frame $\F'=(X',I')$ for which $X'\sub X$, $I'$ is the restriction of $I$ to $X'$, and $X'$ is closed under $I$ in the sense that $I(a)\sub X'$ for all $a\in X'$. This means that $\F'$ is a substructure of $\F$ in which all the $I$-neighbours of an element of $\F'$ belong to $\F'$ as well. It is equivalent to requiring that the inclusion function $X'\hookrightarrow X$ is a bounded morphism from $\F'$ to $\F$. Thus from any model $\M$ on $\F$ we get a model $\M'$ on the inner subframe $\F'$ with $V'(p)=V(p)\cap X'$ such that by \eqref{morphequiv} with $\th$ as the inclusion, 
\begin{equation} \label{submodelsat}
\text{$\M',a\models\ph$ iff $\M,a\models\ph$ for every $a\in X'$ and every formula $\ph$.} 
\end{equation}
An important case of inner subframe is the notion of the  subframe $\F^a=(X^a,I^a)$ of $\F$ \emph{generated by an element $a$}. This has 
$$
X^a=\{b\in X:aI^nb \text{ for some }n\geq 0\},
$$
with $I^a$ being the restriction of $I$ to to $X^a$.
$\F^a$ is often called  \emph{point-generated}.
It is the smallest inner subframe of $\F$ that contains $a$ (when $n=0$). Note when $b\in X^a$ and $bI^nc$, then every member of the associated $I$-path of length $n$ from $b$ to $c$ is in $X^a$, hence $b(I^a)^nc$.
For any $b,c\in X^a$ we have  $aI^nb $ and $aI^m c$ for some $n,m$, hence 
$bI^n_+a I^m_+c$, showing that there is an $I^a_+$-path from $b$ to $c$ in $\F^a$. Thus 
\begin{equation} \label{ptconnected}
\text{any point-generated frame  is \emph{connected}}, 
\end{equation}
a fact we will make use of in the next section.

We turn now to matters of proof theory.
A \emph{normal logic} is any set $\LL$ of formulas that includes all instances of truth-functional tautologies and of the  scheme
\begin{enumerate}
\item[K:]
$\bo(\ph\to\psi)\to(\bo\ph\to\bo\psi)$,
\end{enumerate}
and is closed under the following rules:
\begin{itemize}
\item[]
\emph{modus ponens}: if $\ph,\ph\to\psi\in\LL$, then $\psi\in\LL$;
\item[]
$\bo$-\emph{generalisation}: if $\ph\in\LL$, then $\bo\ph\in\LL$.
\end{itemize}
The members of a logic $\LL$ are called its \emph{theorems}, or more specifically \emph{$\LL$-theorems}. A formula 
$\ph$ is \emph{$\LL$-consistent} when $\neg\ph$ is not an $\LL$-theorem. A set $\Sig$ of formulas is called \emph{$\LL$-consistent} if any finite conjunction $\ph_0\land\cdots\land\ph_{n-1}$ of members of $\Sig$ is $\LL$-consistent.
$\LL$ is \emph{sound for (validity in) a class $\C$} if every $\LL$-theorem is valid in $\C$.
$\LL$ is \emph{complete for (validity in)} $\C$ if, conversely, every formula that is valid in $\C$ is a $\LL$-theorem.
This is equivalent to requiring that every $\LL$-consistent formula is satisfiable in some member of $\C$.
$\LL$ is \emph{strongly complete} for $\C$ if every $\LL$-consistent set of formulas is satisfiable in some member of $\C$.

Any instance of the formula scheme K is valid in any 1-frame, and this validity is preserved by the modus ponens and $\bo$-generalisation rules. Thus for any $\F$, the set of all formulas valid in $\F$ is a normal logic. The smallest (intersection) of all normal logics, also known as K, is sound and strongly complete for validity in the class of all 1-frames. One way to prove that is to use the \emph{canonical frame} $\F_\LL=(X_\LL,I_\LL)$, which can be defined for any $\LL$ by taking $X_\LL$ to be the set of all maximally $\LL$-consistent sets of formulas, with $aI_\LL b$ when $\{\ph:\bo\ph\in a\}\sub b$. The \emph{canonical model} $\M_\LL=(\F_\LL,V_\LL)$ has $V_\LL(p)=\{a\in X_\LL:p\in a\}$ for each variable $p$. It has the property that  any formula $\ph$ is true at an element $a\in X_\LL$ iff $\ph$ belongs to the maximally $\LL$-consistent set $a$, i.e.
\begin{center}
$\M_\LL,a\models \ph$ \enspace iff \enspace $\ph\in a$.
\end{center}
This is sometimes called the Truth Lemma for $\LL$: for details see \cite[Chapter 4]{blac:moda01} or \cite[Chapter 3]{gold:logi92}. Since a formula $\ph$ is a $\LL$-theorem iff it belongs to every member of $X_\LL$, the Truth Lemma implies that the formulas that are true in the model $\M_\LL$ are precisely the $\LL$-theorems. Consequently, any formula that is valid in the frame $\F_\LL$ must be a $\LL$-theorem. Equivalently, every $\LL$-consistent formula is satisfiable in $\M_\LL$, hence satisfiable in $\F_\LL$.
More strongly, any $\LL$-consistent set $\Sig$ of formulas can be extended to a maximally $\LL$-consistent set $a_\Sig$, and so by the Truth Lemma  $\Sig$ is satisfied at $a_\Sig$ in $\M_\LL$, so is satisfiable in $\F_\LL$.

Now letting $\LL$=K, we get the strong completeness of the logic K: every K-consistent set of formulas  is satisfiable in some 1-frame, namely in the model $\M_{\rm K}$ on the frame $\F_{\rm K}$.

\section{Strong completeness of 12g and 8f}   \label{sec:strong}

The canonical frame construction provides a method for showing that a logic $\LL$ is strongly complete for a class $\C$.  Since every $\LL$-consistent set of formulas is satisfiable in $\F_\LL$, it is enough to show that $\F_\LL$ belongs to $\C$. If $\C$ is defined by specified properties, then the  strong completeness proof is reduced to showing that $\F_\LL$ has those properties. There are numerous logics for which the axioms defining $\LL$ can be directly applied to show that  $\F_\LL$ has the properties in question.

The first very general result of this kind appeared in \cite[Section 4]{lemm:intr77}, which defined the axiom scheme
\begin{enumerate}
\item[G$'$:] 
$\di^m\bo^n\ph\to\bo^p\di^q\ph.$
\end{enumerate}
The parameters $m,n,p,q$ are a fixed but arbitrary quadruple of natural numbers. Corresponding to G$'$ is the frame condition
\begin{enumerate}
\item[(g$'$):] 
for all $a,b,c$ such that  $aI^mb$ and $aI^pc$, there exists $d$ with $bI^nd$ and $cI^qd$.
\end{enumerate}
It is proved in Theorems 4.1 and 4.2 of \cite{lemm:intr77} that       
\begin{itemize}
\item 
G$'$ is valid in any 1-frame satisfying (g$'$). 
\item
If a normal logic $\LL$ includes G$'$, then $\F_\LL$ satisfies (g$'$). 
\end{itemize}
The converse of the first item is also true, so in fact a 1-frame validates G$'$ if, and only if, it satisfies (g$'$).

Many interesting cases can be read off from this analysis. Table \ref{table1}  gives the four that we use. The left column gives particular values of the parameters of G$'$, the middle column gives the resulting instance of G$'$ with a name, and the third column  gives the corresponding condition on $I$ expressed by (g$'$). 
\begin{table}[h]
\begin{center}
\begin{tabular}{l|l|l}
$m,n,p,q$ & \quad G$'$ & \quad g$'$  \\  
\hline 
 \\[-2ex]
0,1,0,1 &   D:\enspace $\bo\ph\to\di\ph$  & $I$ is serial  \\  [.5ex]
1,1,0,0 &   B:\enspace $\di\bo\ph\to\ph$         &  $I$ is symmetric \\  [.5ex]
0,2,4,0 &  $4^2$:\enspace  $\bo^2\ph\to\bo^4\ph$  \quad &$aI^4b$ implies $aI^2b$, i.e.\ O1 \\  [.5ex]
0,3,0,0 &  T$^3$:\enspace  $\bo^3\ph\to\ph$   &$aI^3a$, i.e.\ $I^3$ is reflexive
\end{tabular}
\end{center}
\caption{Cases of G$'$}
\label{table1}
\end{table}%

The names D and B are standard. The name $4^2$ is by analogy with the standard name 4 for the axiom $\bo\ph\to\bo^2\ph$ which corresponds to $I$ being transitive. $4^2$ corresponds to $I^2$ being transitive. T$^3$ is by analogy with the name T which corresponds to $I$ being reflexive.

12g can now be defined as the smallest normal logic that includes the axiom schemes D, B and 4$^2$. By the quoted results from  \cite{lemm:intr77}, a 1-frame validates these three axioms iff it is serial, symmetric and satisfies O1, i.e.\ iff it is a quasi-1-plane. The set of all formulas valid in a given quasi-1-plane is a normal logic including D, B and 4$^2$, so it includes 12g as the smallest such logic. Thus 12g is sound for validity in the class of all quasi-1-planes. Moreover the canonical frame $\F_{\rm 12g}$ for 12g is a quasi-1-plane, so 12g is strongly complete for quasi-1-planes. This can be combined with our Corollary \ref{quasipreimage} to yield the stronger

\begin{theorem} \label{strongcomp12g}
$12$g is strongly complete for  the class of all non-degenerate projective $1$-planes.
\end{theorem}
\begin{proof}
Let $\Sig$ be a 12g-consistent set of formulas. Then $\Sig$ is satisfied in the canonical model  $\M_{\rm 12g}$ at some element $a$.  Let $\F$ be the inner subframe of $\F_{\rm 12g}$ generated by $a$. The properties of being serial, symmetric and satisfying O1 are all preserved in passing from $\F_{\rm 12g}$ to $\F$, so $\F$ is  a quasi-1-plane. Moreover, being point-generated it is connected by \eqref{ptconnected}. Therefore by Corollary \ref{quasipreimage} there exists a non-degenerate projective 1-plane $\F^*$ with a surjective bounded morphism $\th:\F^*\to\F$. 

Now the model $\M_{\rm 12g}$ on $\F_{\rm 12g}$ restricts to a model $\M$ on $\F$ in which $\Sig$ is satisfied at $a$
(see \eqref{submodelsat}). So $\Sig$ is satisfiable in $\F$. Hence by Theorem \ref{bmorphsat}, $\Sig$ is satisfiable in $\F^*$ (at some point $a'$ with $\th(a')=a$).

This shows that every 12g-consistent set of formulas is satisfiable in some non-degenerate projective 1-plane, as required.
\end{proof}

This result is a manifestation of the fact that the modal language does not have the expressive power to differentiate the non-degenerate projective 1-planes from other 1-planes, including the elliptic ones, or even from the other quasi-planes:

\begin{theorem}
For any formula $\ph$, the following are equivalent.
\begin{enumerate}[\rm(1)]
\item 
$\ph$ is a theorem of $12$g.
\item 
$\ph$ is valid in all quasi-$1$-planes.
\item 
$\ph$ is valid in all connected quasi-$1$-planes.
\item 
$\ph$ is valid in all\/ $1$-planes.
\item 
$\ph$ is valid in all projective-$1$-planes.
\item 
$\ph$ is valid in all non-degenerate projective-$1$-planes.
\end{enumerate}
\end{theorem}

\begin{proof}
That 1 implies 2 is the soundness of 12g shown above. That 2 implies 3, 3 implies 4, 4 implies 5 and 5 implies 6 are all immediate, since the antecedent class includes the consequent one. That 6 implies 1 follows from the completeness of 12g for non-degenerate projective-$1$-planes given by Theorem \ref{strongcomp12g}.
\end{proof}

This result tell us that if a formula is valid in all projective $1$-planes (5), then it is valid in all $1$-planes (4), and in particular is valid in all the elliptic ones. On the other hand there are formulas that are valid in all elliptic $1$-planes but not in all projective ones. Note first that the elliptic condition O5 is not itself equivalent to the validity of any modal formula. This is readily seen from the fact that the class of 1-frames validating a modal formula is closed under disjoint unions \cite[Theorem 3.14]{blac:moda01}, but the disjoint union of two frames satisfying O5 will not satisfy O5. Nonetheless the modal language can differentiate elliptic structures. For 1-planes a suitable condition is
that $aI^3b$ implies $aI^2b$, which is equivalent to the validity of the scheme 
$\bo^2\ph\to\bo^3\ph$. In the presence of O2 this yields O5, hence ensures there is one $I^2$-equivalence class
\cite[p.17]{pled:some81}. But there is a slightly simpler condition that does not depend on O2:

\begin{lemma} \label{I3refl}
A connected quasi-$1$-plane $\F=(X,I)$  is elliptic iff $I^3$ is reflexive iff there exists an element $a$ with $aI^3a$.
\end{lemma}
\begin{proof}
By Theorem \ref{quasi1props}(4), for any $a\in X$ the subsets $I^2(a)$ and $I^3(a)$ are the only $I^2$-equivalence classes. If $\F$ is elliptic then for any $a$,  $I^2(a)=I^3(a)$, so $aI^3a$ because $aI^2a$. But for the converse it only takes one $a$ to have  $aI^3a$ to get $a\in I^2(a)\cap I^3(a)$ and hence the equivalence classes are equal.
\end{proof}

The logic 8f can be defined as the smallest extension of 12g that includes the axiom scheme T$^3$.  That scheme corresponds to reflexivity of $I^3$ (Table \ref{table1}). Thus 8f is sound for validity in the class of all elliptic quasi-1-planes
\cite[pp.33, 40]{pled:some81}.

\begin{theorem} \label{strongcomp8f}
$8$f is strongly complete for  the class of all non-degenerate elliptic $1$-planes.
\end{theorem}
\begin{proof}
Let $\Sig$ be an 8f-consistent set of formulas. Then by the argument of the proof of Theorem \ref{strongcomp12g},
$\Sig$ is satisfied at some element $a$ in the inner subframe $\F$ of the canonical frame $\F_{\rm 8f}$ generated by $a$, and $\F$ is a connected quasi-1-plane. But now by \cite[4.2]{lemm:intr77},  since T$^3$ is included in 8f,  $\F_{\rm 8f}$ has $I^3$ reflexive. That property is preserved in passing to $\F$, since if $bI_{\rm 8f}cI_{\rm 8f}dI_{\rm 8f}b$ and $b$ is in $\F$, then $c$ and $d$ are in $\F$ with $bIcIdIb$.
Thus $\F$ is an elliptic quasi-$1$-plane by Lemma \ref{I3refl}.
Hence by Theorem \ref{ellipticpreimage}, there exists a non-degenerate elliptic 1-plane $\F^*$ with a surjective bounded morphism $\th:\F^*\to\F$.  So by Theorem \ref{bmorphsat}, $\Sig$ is satisfiable in $\F^*$.
\end{proof}

Another property of a relation $I$ that it is natural to consider is \emph{irreflexiveness}, which is well known not to be equivalent to the validity of any modal formula (e.g.\ it is not preserved by surjective bounded morphisms). For a 2-frame $\F$, irreflexiveness is built in since $I\sub P\times L$ and $P\cap L=\emptyset$, so $aIa$ never occurs. Hence the associated symmetric 1-frame $\F_+$ is irreflexive, and indeed never has $a(I_+)^n a$ if $n$ is odd. Thus in applying Corollary \ref{quasipreimage} to the proof of Theorem \ref{strongcomp12g}, what is produced is an irreflexive 1-plane satisfying a 12g-consistent set. So 12g is strongly complete for the class of irreflexive non-degenerate elliptic $1$-planes.

In the real elliptic plane, the conjugacy relation is irreflexive, since no point lies on its polar line, so $aIa$ cannot hold. Irreflexivity of $I$ is one of the axioms for elliptic plane geometry of Kordos \cite{kord:elli73}.
Now in the proof of Theorem \ref{ellipticpreimage}, the graph $\F_0$ is irreflexive and the operations for repairing defects do not introduce any reflexive points, so the non-degenerate elliptic 1-plane constructed in that theorem is irreflexive. It follows by the proof of Theorem \ref{strongcomp8f}, that $8$f is strongly complete for  the class of all irreflexive non-degenerate elliptic $1$-planes.

\section{Finite model property}  \label{sec:fmp}

A modal logic $\LL$ has the \emph{finite model property} when every $\LL$-consistent formula is satisfiable in some model on a \emph{finite} frame that validates $\LL$. Equivalently this means that $\LL$ is complete for validity in the class of all finite frames that validate $\LL$  (hence sound and complete for this class).

We will show that the logics 12g and 8f have the finite model property. Since they have finitely many axioms, it follows by standard theory that the property of being a theorem of the logic is algorithmically decidable for each of them. For an explanation of this theory see \cite[\S6.2]{blac:moda01} or
 \cite[\S16.2]{chag:moda97}.

We use the well known \emph{filtration} method of reducing a model to a finite one. Let $\LL$ be a normal logic extending 12g. Fix a formula $\ph$ that is $\LL$-consistent. Then by the construction in the first part of the proof of Theorem \ref{strongcomp12g}, $\ph$ is true at some element of a model $\M$ on a connected quasi-1-plane $\F=(X,I)$ which is a point-generated subframe of the canonical frame $\F_\LL$.
Let $\Phi$ be the finite set of all subformulas of $\ph$. For each $a\in X$, let 
$ \Phi_a=\{\psi\in\Phi: \M,a\models\psi\}$. 
Define an equivalence relation $\sim$ on $X$ by putting
$a\sim b$ iff $\Phi_a=\Phi_b$.
Then with $\ab{a}=\{b\in X:a\sim b\}$ being the $\sim$-equivalence class of $a$, we put  $X_\Phi=\{\ab{a}:a\in X\}.$
The set $X_\Phi$ is finite, because the map $\ab{a}\mapsto \Phi_a$ is a well-defined injection of $X_\Phi$ into the powerset of $\Phi$. Thus $X_\Phi$ has size at most $2^n$, where $n$ is the size of $\Phi$.

 A \emph{filtration of $\M$ through $\Phi$} is any model $\M'=(\F',V')$ with $\F'$ of the form $(X_\Phi,I')$ such that 
 \begin{enumerate}[(i)]
\item 
$aIb$ implies $\ab{a}I'\ab{b}$, i.e.\ $a\mapsto\ab{a}$ is a homomorphism $\F\to\F'$.
\item
if $\ab{a}I'\ab{b}$, then for all $\bo\psi\in\Phi$, $\M,a\models\bo\psi$ implies  $\M,b\models\psi$.
\item
$V'(p)=\{\ab{a}\in X_\Phi: a\in V(p)\}$ for all variables $p\in\Phi$.
\end{enumerate}
For such an $\M'$, we have the \emph{Filtration Theorem}: for any $\psi\in\Phi$ and $a$ in $\F$,
\begin{equation*}
\text{$\M,a\models\psi$\enspace  iff\enspace  $\M',\ab{a}\models\psi$.}
\end{equation*}
This is shown by  induction on the formation of the formula $\psi\in\Phi$: see \cite[Theorem.~2.39]{blac:moda01} or \cite[Theorem.~5.23]{chag:moda97}.

Filtrations of $\M$ through $\Phi$ do exist. The \emph{least} filtration has $V'(p)$ defined by (iii) if $p\in\Phi$ and
$V'(p)=\emptyset$ (or anything) otherwise; while $I'$ is defined for $\a,\b\in X_\Phi$ by putting
\begin{equation}\label{defnI}
\text{$\a I'\b$ iff there exist $a\in\a$ and $b\in\b$ such that $aIb$.}
\end{equation}
This definition makes $I'$ symmetric whenever $I$ is, and serial whenever $I$ is.

We are now in a position to demonstrate that 8f has the finite model property.
Put $\LL$=8g in the above and let $\M'$ be the least filtration of $\M$ through $\Phi$. Now the underlying frame $\F$ of $\M$ is a subframe of $\F_{\rm 8f}$ that is an elliptic quasi-1-plane by the argument in the proof of Theorem \ref{strongcomp8f}, so $I$ is symmetric and satisfies O5. Hence $I'$ is symmetric, as just noted. It also satisfies O5, for if $\ab{a},\ab{b}\in X_\Phi$, then $aIcIb$ for some $c$ as $\F$ satisfies O5, hence $\ab{a}I'\ab{c}I'\ab{b}$ by the filtration homomorphism property (i). Thus $\F'$ is an elliptic quasi-1-plane and therefore validates 8f. But $\ph$ is true at some point $a$ of $\M$, and $\ph\in\Phi$, so $\ph$ is true at $\ab{a}$ in $\M'$ by the Filtration Theorem. This proves that any 8f-consistent formula is satisfiable in a model on a finite frame validating 8f, as required.

In the case that $\LL$=12g, we  only know so far that the subframe $\F$ of $\F_{\rm 12g}$ is a connected quasi-1-plane that may be projective or elliptic. If it is elliptic, we can proceed exactly as in the preceding paragraph to obtain a satisfying model for $\ph$ on a finite elliptic quasi-1-plane. If however $\F$ is projective, we proceed to construct a suitable 
$\M'$ in a different way, with new definitions of $X_\Phi$ and $I'$ as follows.  $\F$ has two $I^2$-equivalence classes, which we will label $P$ and $L$. We filtrate $P$ and $L$ separately through $\Phi$. For each $a\in P$, let $\ab{a}=\{ b\in P:a\sim b\}$, while if $a\in L$ let $\ab{a}=\{ b\in L:a\sim b\}$. Put $P'=\{\ab{a}:a\in P\}$, $L'=\{\ab{a}:a\in L\}$, and $X_\Phi=P'\cup L'$. Since $P$ and $L$ are disjoint, so too are $P'$ and $L'$, which  each have at most $2^n$ elements. Hence $X_\Phi$ is of size at most $2^{n+1}$. Then $I'$ is defined for $\a,\b\in X_\Phi$ by the condition  \eqref{defnI}. To complete the definition of $\M'$ we define $V'(p)$ by (iii) if $p\in\Phi$, and put $V'(p)=\emptyset$ otherwise, as before.

$\M'$ is a filtration of $\M$ through $\Phi$: (iii) holds by definition of $V'$, and  reading  \eqref{defnI} from right to left shows that $aIb$ implies  $\ab{a}I'\ab{b}$, i.e.\ (i) holds. For (ii), suppose  $\ab{a}I'\ab{b}$. Then $a'I'b'$ for some $a'\in\ab{a}$ and $b'\in\ab{b}$, so if $\bo\psi\in\Phi$ and  $\M,a\models\bo\psi$, then $\M,a'\models\bo\psi$ as $a\sim a'$, hence $\M,b'\models\psi$ as $a'Ib'$, and  $\M,b\models\psi$ as $b\sim b'$, which proves (ii). It follows that the Filtration Theorem holds for $\M'$, leading us to conclude that the 12g-consistent formula $\ph$ is satisfiable in the finite model $\M'$.

It remains to show that $\F'=(X_\Phi,I')$ validates 12g. As noted above,  \eqref{defnI} ensures that the properties of being serial and symmetric are preserved in passing from $I$ to $I'$. We show that $I'$ satisfies O1.
First, by parts (4) and (5) of Theorem \ref{quasi1props}, whenever $aIb$ then one of $a$ and $b$ belongs to $P$ and the other belongs to $L$. Hence by  \eqref{defnI},  whenever $\a I'\b$ then one of $\a$ and $\b$ belongs to $P'$ and the other belongs to $L'$. Iterated application of this gives that if $\a(I')^n\b$ and $n$ is even, then $\a$ and $\b$ both belong to $P'$, or both belong to $L'$.  We can now prove O1:  suppose $\a(I')^4\b$. If $\a\in P'$, then $\b\in P'$, so choosing $a\in\a$ and $b\in\b$ we have $a,b\in P$, hence $aI^2b$ as $P$ is an $I^2$-equivalence class. But then $\ab{a}(I')^2\ab{b}$ as in the 8f case, i.e.\  $\a(I')^2\b$. But if $\a\in L'$, then $\b\in L'$ and similarly we infer that $\a(I')^2\b$.

This proves that $\F$ satisfies O1 and is a quasi-1-frame, hence validates 12g. (In fact it is projective, with $P'$ and $L'$ being its only two $(I')^2$-equivalence classes.) That completes the proof of the finite model property for 12g.

We end with a comment on irreflexiveness and the logic  8f. It was shown
at the end of Section \ref{sec:strong} that 8f is  (strongly) complete for the class of all \emph{irreflexive} non-degenerate elliptic $1$-planes.  However it does not have the finite model property for this class.  Pledger noted in  \cite[p.24]{pled:some81} that the class has \emph{no} finite members at all, by the Friendship Theorem of graph theory. That theorem, due to Erd\H os, R\'enyi and S\'os, can be formulated as saying that a finite irreflexive symmetric 1-frame in which any two distinct elements have exactly one common neighbour must be a \emph{windmill} graph,   a set of triangles that all have one common vertex, while no two triangles share any other vertex (see \cite[Chapter 44]{aign:proo18} for an exposition).  Thus any finite irreflexive elliptic 1-plane is a windmill. But no windmill can satisfy the non-degeneracy condition O4$'$, as the reader may like to confirm.

%For, if $aIbIc$ with $a\ne c$ and not $aIc$, then $b$ can only be the common vertex, which has all other vertices as neighbours, so there is no $d\ne b$ such that not $bId$.

%%%%%%%%%%%%%%%%%%%%%%%%%%%%%%%%%%%%%%%%%%%%%%%%%%%%

\bigskip\newpage
\noindent {\bf{Acknowledgement}}

\medskip\noindent
My thanks to Shirley Pledger for  making  \cite{pled:some81} available and  providing  biographical information.

%%%%%%%%%%%%%%%%%%%%%%%%%%%%%%%%%%%%%%%%%%%%%%%%%%%% 

%%References 

%\newpage
\bibliographystyle{plain}

\end{document}